\documentclass{icgg}

\usepackage{graphicx}     
\usepackage{microtype}    
                          
\usepackage{amsmath,amsthm,amssymb}
\usepackage{bm}                         

\title{\uppercase{Double-Line Rigid Origami}}

\author{Thomas C. HULL\textsuperscript{1}
  and
  Tomohiro TACHI\textsuperscript{2}}
\affiliation{
  \textsuperscript{1} Western New England University, USA
  \qquad
  \textsuperscript{2} University of Tokyo, Japan}

\hypersetup{
  pdfauthor = {Tomohiro Tachi and Thomas Hull},
  pdftitle = {Instructions for Formatting Papers for the 16th
    International Conference on Geometry and Graphics},
  pdfsubject = {LaTeX template file},
  pdfkeywords = {Geometry, graphics, ISGG, ICGG, LaTeX, article
    formatting}
}

{\makeatletter
 \gdef\xxxmark{%
   \expandafter\ifx\csname @mpargs\endcsname\relax 
     \expandafter\ifx\csname @captype\endcsname\relax 
       \marginpar{xxx}
     \else
       xxx 
     \fi
   \else
     xxx 
   \fi}
 \gdef\xxx{\@ifnextchar[\xxx@lab\xxx@nolab}
 \long\gdef\xxx@lab[#1]#2{\textbf{[\xxxmark #2 ---{\sc #1}]}}
 \long\gdef\xxx@nolab#1{\textbf{[\xxxmark #1]}}
}

\newtheorem{theorem}{Theorem}

\theoremstyle{definition}

\begin{document}

\twocolumn[%
\begin{@twocolumnfalse}
  \maketitle
  \begin{abstract}

In this paper, we will show methods to interpret some rigid origami with higher degree vertices as the limit case of structures with degree-4 supplementary angle vertices. 
The interpretation is based on separating each crease into two parallel creases, or \emph{double lines}, connected by additional structures at the vertex. We show that double-lined versions of degree-4 flat-foldable vertices possess a rigid folding motion, as do symmetric degree-$2n$ vertices. 
The latter gives us a symbolic analysis of the original vertex, showing that the tangent of the quarter fold angles are proportional to each other. 
The double line method is also a potentially useful in giving thickness to rigid origami mechanisms. 
By making single crease into two creases, the fold angles can be distributed to avoid $180^\circ$ folds, when panels can easily collide with each other. 
This can be understood as an extension of the crease offset method of thick rigid origami with an additional guarantee of rigid-foldability.
  \end{abstract}
  \keywords{Geometry, graphics, ISGG, ICGG, LaTeX, article formatting}
\end{@twocolumnfalse}]


\section{Introduction}
\subsection{Background}
Rigid origami is a kinematic model of folding a \emph{crease pattern} $C$, a straight-line planar graph drawn on a compact region of the plane, into a 3D shape such that the faces of $C$ remain flat and the edges of $C$, also called the \emph{creases}, act like hinges.
Some families of rigid origami crease patterns exhibit 1-degree of freedom (DOF) mechanisms applicable to deployable structures. 
One of the families of 1-DOF origami is made with a network of degree-4 vertices whose sector angles are supplementary to each other, aka \emph{flat-foldable} vertices. 
Such a system has some interesting and useful characteristics:
\begin{enumerate}
\item \label{flattens} The structure has two flat states (flat-unfolded and flat-folded).
\item \label{stateimplies}If there exists a valid 3D state, then there exists a continuous path from the unfolded to folded states~\cite{Tachi:2009}.
\item \label{tanhalf}The tangent of half of each fold angle at the creases are proportional to each other. The coefficient is a constant determined only by the sector angles in the crease pattern. (See the Appendix for details.)
\end{enumerate}
By property~\ref{tanhalf}, the configuration space of this type of pattern may be analyzed symbolically, while in generic origami patterns even a single vertex model often requires numerical analysis to study. 
Also, because of property~\ref{stateimplies}, the vertices link to form cycles of constraints, forming an over-constrained system (a system with more constraints than variables), but we still obtain a foldable mechanism (an over-constrained mechanism).
This over-constrained nature is useful for designing systems with structural stiffness and deployability, allowing for controlled and uniform actuation~\cite{filipov-etal:2015}.
The drawback is that this powerful property is limited to the geometry of degree-$4$ flat-foldable vertices.
Also, property~\ref{flattens} is a useful property that the structure can be stowed in a compact, flat state; 
however, when we use actual thick material, the mechanism does not completely go flat, so an idea to overcome the thickness is required for making use of flat-foldability of degree-4 flat-foldable vertex mechanism.

In this paper we will describe a method for converting a crease pattern $C$ into a flat-foldable crease pattern $DL(C)$, called the \emph{double-line version of $C$}, consisting only of degree-4 vertices and whose rigid origami kinematics is identical to that of $C$.
This method can interpret the kinematics of  some higher degree or non flat-foldable origami as the combination of degree-4, flat-foldable vertices, so that it can help designers of rigid origami mechanisms to bring the advantages of degree-4, flat-foldable origami to more complicated crease patterns.
Also, the double-line crease pattern is itself has practical advantage as the method for dealing with thickness.
The method can split the fold angle of a single line into smaller fold angles of two lines, to make full range motion of the original flat-foldable crease pattern without being interfered by the collision of thick panels.

We follow the usual definitions of rigid origami \cite{bel-hull:2002,LangTwist15,Tachi-Hull:2016}.  A \emph{rigid folded state} of a crease pattern $C$ is a piecewise isometric homeomorphism $\sigma:P\to\mathbb{R}^3$  where P is the compact region on which $C$ is embedded and each face $F$ of $C$ is congruent to its image $\sigma(F)$.  The \emph{fold angle} of a crease is the signed angle formed by the normals of the images of the incident faces under $\sigma$.  Thus, zero fold angle means the crease is unfolded and a \emph{valley} (\emph{mountain}) crease has positive (negative) fold angle.  Each rigid folded state is determined by the fold angles at each crease, and thus the \emph{parameter space} of $C$ is $\mathbb{R}^{|E|}$ where $E$ is the set of creases (i.e., edges) of $C$.  The \emph{configuration space} of $C$ is the set of points in the parameter space that represent a valid rigid folded state.  A \emph{rigid folding motion} of an origami crease pattern $C$ is a path in its configuration space.

The mountain and valley creases of a degree-4 vertex in a rigid folded state must follow a pattern:  There will be either 3 Vs and 1 M (or vice-versa), and the sector angles adjacent to the lone M crease must sum to $<180^\circ$.  This means that a generic degree-4 vertex that has been rigidly folded will have at least two different possible locations for the lone M crease, as can be seen in Figure~\ref{fig:ffd4vert}.  We refer to these two possibilities as different \emph{modes} in which the vertex can fold.  (Note that if $\alpha=\beta$ in Figure~\ref{fig:ffd4vert} then there are three possible locations for the lone M fold, but only one case, where the M is between the sector angles $\alpha$ and $\beta$, results in a rigid folding motion where all four creases are folding with non-zero folding angles.)  Furthermore, Equation~\eqref{eq:folding-modes} in the Appendix indicates that the opposite pair of valley creases, which we call the \emph{major creases}, will have equal fold angles and fold faster than the mountain-valley opposite pair (which are complementary and called the \emph{minor creases}).


\subsection{The Double-Line Technique}

We begin by specifying the double-line technique for a degree-4 vertex; this process is illustrated in Figure~\ref{fig:DL}. Given a degree-4 (not necessarily flat-foldable) origami vertex $V$ with sector angles $\alpha$, $\beta$, $\gamma$, and $\delta$ between the creases $e_0, \ldots, e_3$, we place line segments perpendicular to each crease line $e_i$ at a distance $r_i$ from $V$, for $i=0,\ldots, 3$, so as to make a polygon 
 $P=c_0 c_1 c_2 c_3$ (see Figure~\ref{fig:DL}(b)).  Then for each crease $e_i$ we draw two lines parallel to $e_i$ emanating from the polygon corners $c_{i-1}$ and $c_i$, where the indices are taken mod 4.  These parallel creases (the ``double-line" of each crease $e_i$) together with the polygon $P$ gives us a new crease pattern, called the \emph{double-line of $V$}, which we denote $DL(V)$ and is shown in Figure~\ref{fig:DL}(c).  Note that the sector angles $\alpha$, $\beta$, $\gamma$, and $\delta$ are preserved, although separated among the vertices of $DL(V)$, so that each vertex will have sector angles, e.g., $\alpha, 90^\circ, \pi-\alpha$, and $90^\circ$, in order.  This guarantees that each vertex of $DL(V)$ will be flat-foldable, even if the original vertex $V$ is not.

\begin{figure}[t]
	\includegraphics[width=\linewidth]{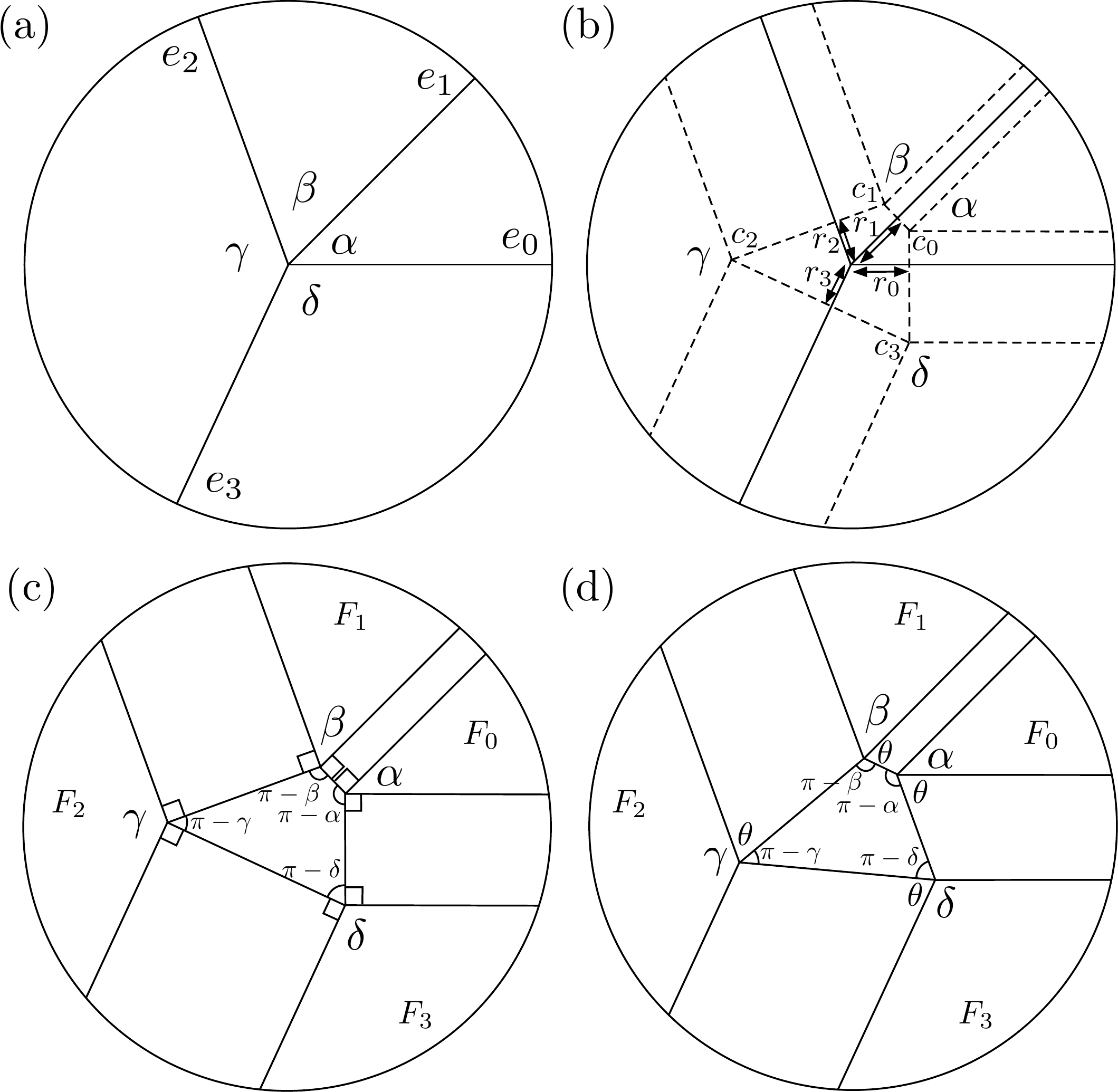}
	\centering
	\caption{Constructing a double-line vertex.  (a) The original origami vertex $V$.  (b) Constructing the polygon $P$ and the double lines.  (c) The new crease pattern $DL(V)$.  (d) The rotated double-line crease pattern $DL(V,\theta)$.}
	\label{fig:DL}
\end{figure}

The crease pattern $DL(V)$ always has the ``double-line" creases being perpendicular to the central polygon $P$.  However, we may rotate the central polygon so each side of $P$ makes an angle of $\theta$ with the double-line creases, as shown in Figure~\ref{fig:DL}(d).  
We refer to this crease pattern as the \emph{rotated double-line of $V$} and denote it $DL(V,\theta)$.  Our rotation of the central polygon does not change the angles $\alpha$, $\beta$, $\gamma$, and $\delta$ or their supplements.  Thus the vertices of $DL(V,\theta)$ will all be flat-foldable as well.

This crease pattern construction method can be performed on any origami vertex in the same way.  If the vertex $V$ is of degree $n$, then the central polygon of $DL(V)$ will have $n$ sides, there will be $n$ pairs of creases emanating from this polygon, and by the same argument as above, each vertex of $DL(V)$ will be flat-foldable.  We may also apply this method to crease patterns $C$ with more than one vertex, where we adjust the lengths $r_i$ in Figure~\ref{fig:DL}(b) for each vertex in order to make the double-line creases line up between adjacent vertices in $C$.

This idea of creating new crease patterns by replacing single creases with parallel creases (or \emph{pleats}, as they are called in the origami literature) is not new \cite{bateman:2002,lang-bateman:2011,Palmer:1996,Resch:1968,Verrill:1998}.  Our concern, however, is the rigid foldability and the kinematics of double-line crease patterns.
Specifically, we are interested in the comparison between the configuration space of $DL(V,\theta)$ and $V$, from the viewpoint that a rigid folded state of $DL(V,\theta)$ is the ``chamfered'' version of a folded shape of $V$.
Here, we consider the configuration of $V$ \emph{corresponding to} a configuration of $DL(V,\theta)$, which is defined by a folded shape of $V$ whose fold angle of each crease is equal to the sum of fold angles of the corresponding pair of fold lines in $DL(V,\theta)$.

\textbf{Overview:}
We will see in Section~\ref{sec:deg4} that if the original vertex $V$ is flat-foldable, then there will always exist rigid folding motions for $DL(V,\theta)$; we will classify these and compare their configuration space with those of $V$. 
We show that by the careful choice of $\theta$, the folding speed ratio between the pair of creases can be tweaked, and this also enables the configuration space of $V$ corresponding to $DL(V,\theta)$ can span fully, i.e., between flat-folded states.
In Section~\ref{sec:2n} we will describe a class of rigid folding motions for double-line crease patterns of symmetric, degree-$2n$ vertices.
We also show the number of possible modes that are realized by double-line origami.  
In Section~\ref{sec:thick} we will describe how double-line crease patterns can be applied to the problem of \emph{thick origami}, or rigid folding motions using material with non-trivial thickness.

\section{Double-line degree-4 vertices}
\label{sec:deg4}

\begin{figure*}[tbhp]
	\includegraphics[width=0.9\linewidth,page=1]{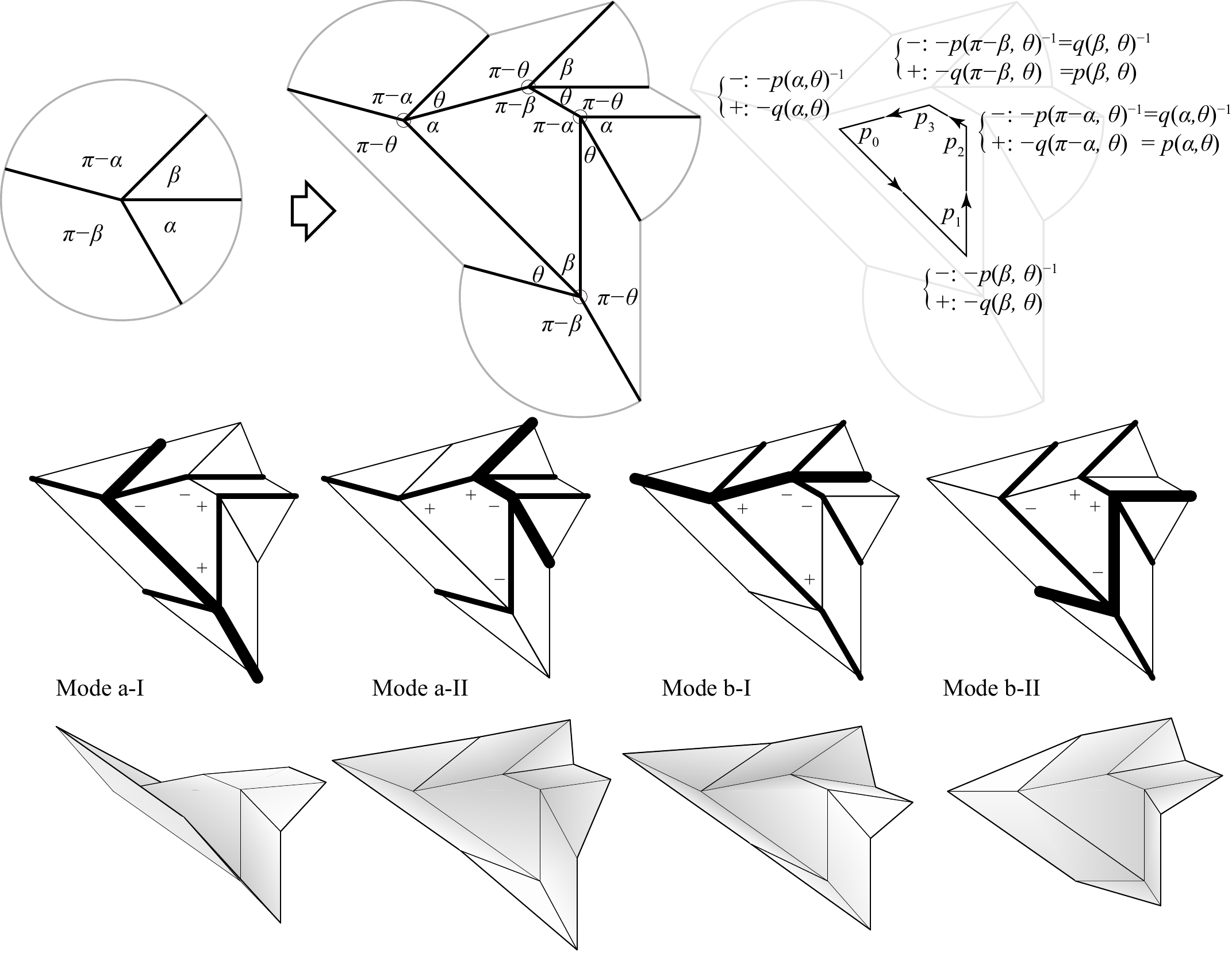}
	\centering
	\caption{Four folding modes of double-lined flat-foldable degree-4 vertex. Mode $-$ and $+$ indicates that the absolute speed of the folding increases (multiplied by $\pm p^{-1}$ or $\pm q^{-1}$)  or decreases (multiplied by $\pm p^{+1}$ or $\pm q^{+1}$). The thicker lines have higher absolute folding speed and the thinner lines have lower folding speed.}
	\label{fig:ffd4-modes}
\end{figure*}

\begin{figure*}[tbhp]
	\includegraphics[width=1.0\linewidth,page=2]{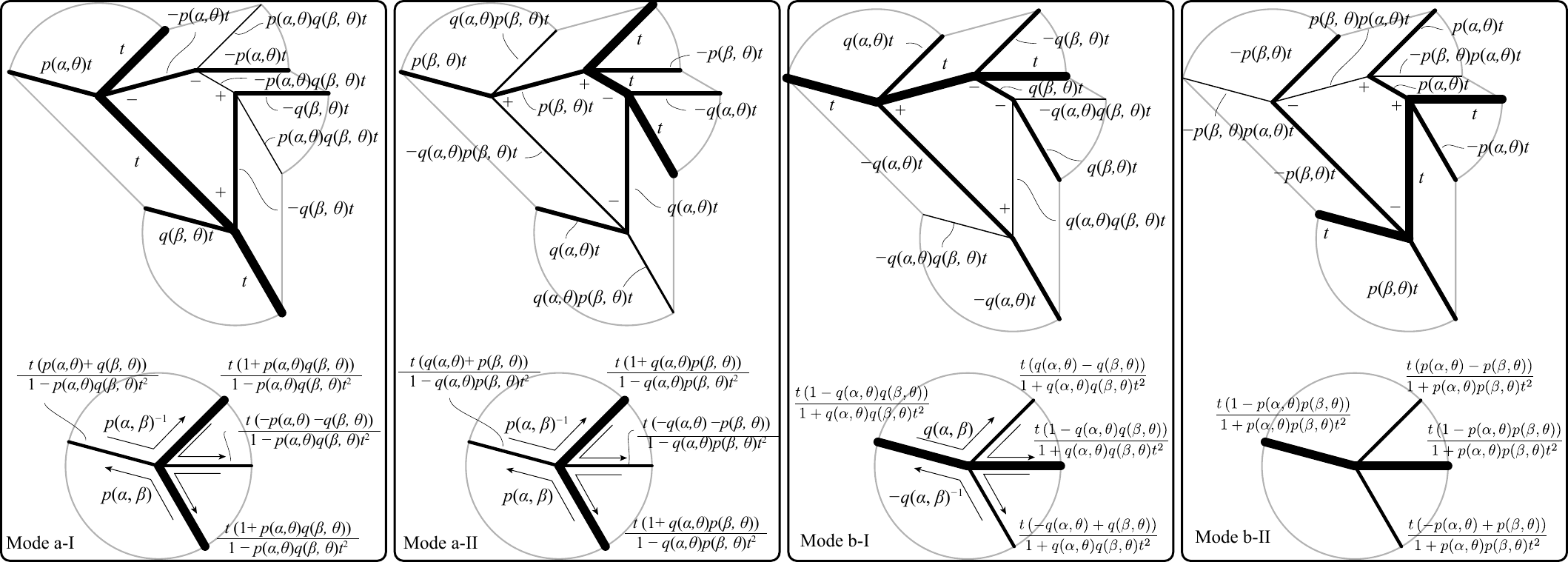}
	\centering
	\caption{Four folding modes and their correspondence to the folding modes of original vertex.}
	\label{fig:ffd4-summation}
\end{figure*}

In this section, we prove that there exists four rigid folding modes for $DL(V,\theta)$ for an arbitrary twist angle $\theta$ if $V$ is a flat-foldable degree-$4$ vertex (Theorem~\ref{thm:deg4ff}).
Also, we further compare the configuration space of $DL(V)$ with that of $V$, to see that two of these four modes of $DL(V,\theta)$ fully represent one of the folding modes of $V$, and of the other two modes of $DL(V,\theta)$ partially represent the other folding mode of $V$ (Theorem~\ref{thm:deg4ff-represent}).
By tweaking angle $\theta$, we can always make one of the modes to be fully represented, and also in a special case, we may additionally make the minor crease of $V$ to be represented by a pair of creases with the same fold angle (the half angle of the original).

\begin{theorem}\label{thm:deg4ff}
If $V$ is a flat-foldable, degree-4 vertex, then $DL(V,\theta)$ has a rigid folding motion from the flat, unfolded state.  Furthermore, there are at least four folding modes for $DL(V,\theta)$ as shown in Figure~\ref{fig:ffd4-modes}.
\end{theorem}
\begin{proof}
Let $V$ be a flat-foldable, degree-4 vertex.  Then there must be two adjacent sector angles in $V$ whose sum is $\leq 180^\circ$; call these $\alpha$ and $\beta$.  The flat foldability of $V$ then tells us that the sector angles are $\alpha$, $\beta$, $\pi-\alpha$, and $\pi-\beta$.  
Refer to Figure~\ref{fig:ffd4-modes}.
Every vertex of $DL(V,\theta)$ is a flat-foldable, degree-4 vertex, so if the pattern has a continuous rigid folding motion, the tangent of half of the fold angles of all the creases are proportional to each other~\cite{evans:2015,Tachi-Hull:2016} (see Appendix).
This proportion, the \emph{folding speed coefficient}, between incident creases is only determined by the sector angles of the incident vertex, so the rigid foldability of the whole system can be checked by determining if,
\emph{for each strictly interior face, the product of the  speed coefficients between adjacent edge equals $1$} (i.e., Equation~\eqref{eq:prod}).
Now, the  speed coefficients at each corner in counterclockwise order are:

\begin{align}
p_0=
&\begin{cases}
-p(\alpha,\theta)^{-1} & \textrm{case $-$}\\
-q(\alpha,\theta) & \textrm{case $+$}\\
\end{cases}
\\
p_1=
&\begin{cases}
-p(\beta,\theta)^{-1} & \textrm{case $-$}\\
-q(\beta,\theta) & \textrm{case $+$}\\
\end{cases}
\\
p_2=
&\begin{cases}
q(\alpha,\theta)^{-1} & \textrm{case $-$}\\
p(\alpha,\theta) & \textrm{case $+$}\\
\end{cases}
\\
p_3=
&\begin{cases}
q(\beta,\theta)^{-1} & \textrm{case $-$}\\
p(\beta,\theta) & \textrm{case $+$}\\
\end{cases},
\end{align}
where $p$ and $q$ are the function of sector angels given in Appendix.
For arbitrary $\alpha, \beta, \theta$, four patterns of modes $(-++-)$, $(+--+)$, $(++--)$, $(--++)$  satisfy Equation~\eqref{eq:prod}, i.e., $p_0p_1p_2p_3 =1$.
We respectively call these modes a-I, a-II, b-I, and b-II


\end{proof}

\begin{figure}[t]
	\includegraphics[width=\linewidth]{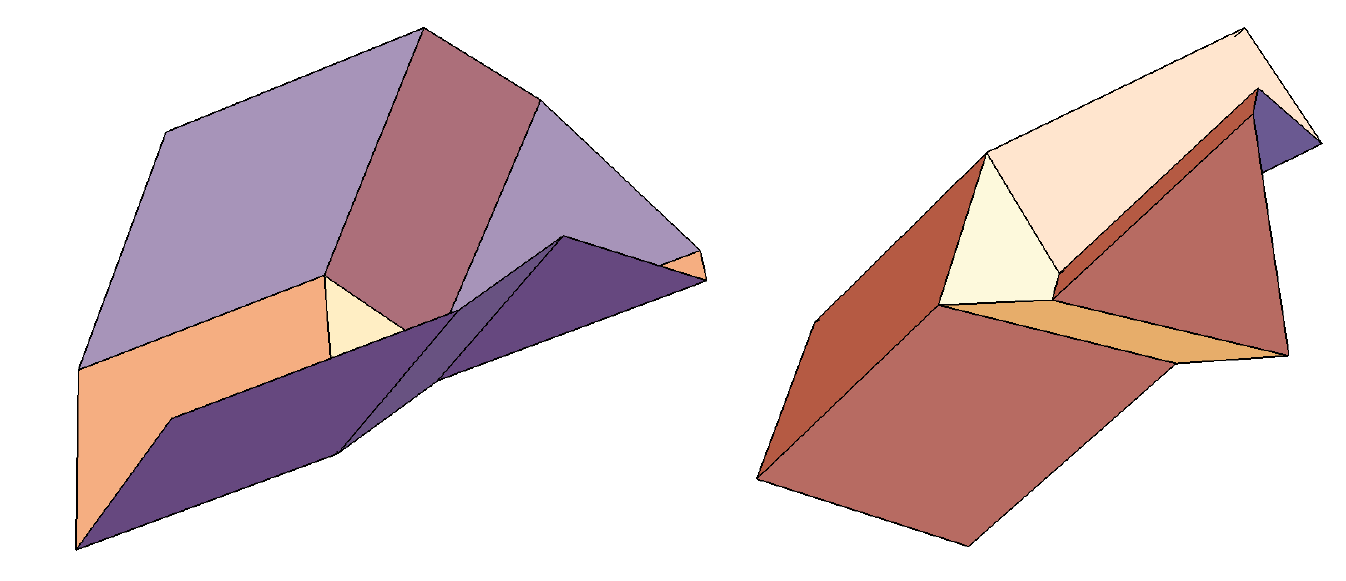}
	\centering
	\caption{The symmetric folding modes $(+ - + -)$ and $(- + - +)$ when $\alpha=\beta$ and $\theta=90^\circ$.}  
	\label{fig:othermodes}
\end{figure}

If more symmetry is introduced then other modes are possible.  Specifically, if $\alpha=\beta$ or $\alpha=\pi-\beta$ and $\theta=90^\circ$ then the modes $(+ - + -)$ and $(- + - +)$ are possible and give different folding results, as shown in Figure~\ref{fig:othermodes}.


Of particular interest is the comparison of the kinematics of $DL(V,\theta)$ with those of $V$.  That is, if $V$ is a degree-4, flat-foldable vertex, then the relationship between its fold angles as $V$ rigidly folds and unfolds are as given in the Appendix.  In $DL(V,\theta)$ the regions labeled $F_0,\ldots, F_3$ in Figure~\ref{fig:DL}(c) and (d) correspond to the sector regions of $V$, and thus we would hope that as $DL(V,\theta)$ rigidly folds and unfolds the angles between the planes $F_0,\ldots, F_3$ will be the same as the folding angles of $V$.  This is indeed the case, as we will now show.  

In $DL(V,\theta)$ one set of double line creases on opposite sides of the inner polygon will correspond to the major creases (the \emph{major axis}) and the other set will correspond to the minor creases (the \emph{minor axis}).  Since the major creases of a degree-4 vertex fold faster then the minor creases, the major axis of $DL(V,\theta)$ will be the double line creases that include the fastest folding speed, which includes the creases with folding angle parameterization  $t=\tan(\rho/2)$, where the actual folding angle is $\rho$.  These are the creases labeled $t$ in the parameterizations shown in Figure~\ref{fig:ffd4-summation} for the four different modes that are possible, where we also show the folding multipliers at each double-line crease to indicate how they are proportional to one another; these multipliers are obtained from the folding kinematics of the degree-4 vertices in $DL(V,\theta)$. 

Consider Mode a-I in Figure~\ref{fig:ffd4-summation}.  (The other modes follow similarly.)  The sum of the folding angles of the major axis creases is
\begin{multline}\label{eq:majsum}
2\arctan(t)+2\arctan(p(\alpha,\theta)q(\beta,\theta)t)\\
= 2\arctan\left(
\frac{(1+p(\alpha,\theta)q(\beta,\theta))t}
{1-p(\alpha,\theta)q(\beta,\theta)t^2}\right).
\end{multline}
The sum of the folding angles of the minor axis creases is
\begin{multline}\label{eq:minsum}
2\arctan(p(\alpha,\theta)t) + 2\arctan(q(\beta,\theta)t)\\
= 2\arctan\left(
\frac{(p(\alpha,\theta)+q(\beta,\theta))t}
{1-p(\alpha,\theta)q(\beta,\theta)t^2}\right).
\end{multline}
Dividing arguments of \eqref{eq:majsum} and \eqref{eq:minsum} gives us
$$\frac{\cos\frac{\alpha+\beta}{2}}{\cos\frac{\alpha-\beta}{2}} =        
\frac{1-\tan\frac{\alpha}{2}\tan\frac{\beta}{2}}{1+\tan\frac{\alpha}{2}\tan\frac{\beta}{2}}   $$
which is the vertex multiplier $p(\alpha,\beta)$ as seen in the Appendix.  Thus the kinematics of the major and minor axes of $DL(V,\theta)$ will be the same as that of the degree-4, flat-foldable vertex $V$ when folding in Mode a-I in Figure~\ref{fig:ffd4-summation}, and the other modes can be shown similarly.

However, the fact that the kinematic equations for $V$ and $DL(V,\theta)$ are the same does not mean the entire range of their folding motions will be the same; indeed, they are not in general.  Notice that if $\beta<\theta$ then $q(\beta,\theta)>0$ and when $\beta>\theta$ we have $q(\beta,\theta)<0$.  Similarly, $p(\alpha,\theta)>0$ when $\theta<\pi-\alpha$  and $p(\alpha,\theta)\leq 0$ otherwise.  

\begin{figure}[t]
	\includegraphics[width=0.6\linewidth]{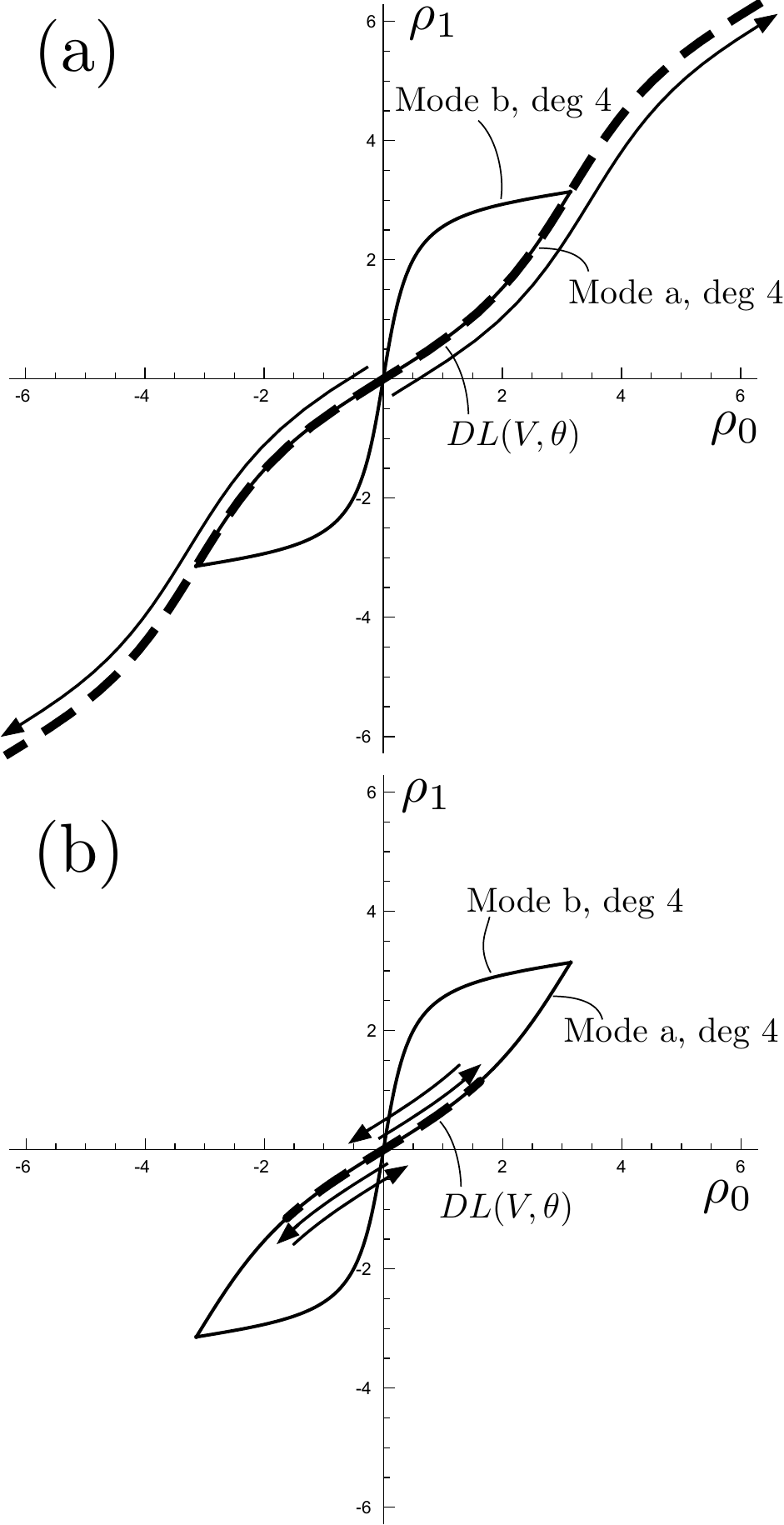}
	\centering
	\caption{The configuration spaces of the folding angles $(\rho_0,\rho_1)$ of a degree-4 flat-foldable vertex $V$ under Mode a-I, with the equivalent configuration space of $DL(V,\theta)$ superimposed. (a) The case $\beta<\theta<\pi-\alpha$. (b) The case $\theta>\pi-\alpha$, $\beta$.}  
	\label{fig:config}
\end{figure}

Now, if $p(\alpha,\theta)$ and $q(\beta,\theta)$ have the same sign, then Equations~\eqref{eq:majsum} and \eqref{eq:minsum} tell us that the fold angles at each double-line crease on the major and minor axes will have the same sign, and so they will add to give us a total fold angle in the range $[-2\pi, 2\pi]$; see Figure~\ref{fig:config}(a). This happens if $\beta<\theta<\pi-\alpha$ or $\pi-\alpha<\theta<\beta$.  In this case the configuration space of the rigid folding of $DL(V,\theta)$ will contain the configuration space of the rigid folding of $V$, and thus $DL(V,\theta)$ will mimic the folding of $V$ but then will continue to fold with double-line folding angle sums greater then $\pi$ in absolute value, until collision of the material occurs.

On the other hand, if  $p(\alpha,\theta)$ and $q(\beta,\theta)$ have opposite signs, then the fold angles at each double-line crease will subtract from each other.  In particular, this means that if we let $S$ be the sum of the fold angles at a double-line pair (either the major or minor axis), then we will have $S=0$ when the folding angles equal zero at the unfolded state and $S=0$ again when the folding angles all equal $\pm\pi$ at the flat-folded state.  In between the sum $S$ will increase and then decrease, or vice-versa, achieving a maximum (or minimum) value $M$.  The configuration space for this is shown in Figure~\ref{fig:config}(b).  If this extreme value $M$ is $\pm\pi$ then $DL(V,\theta)$ will mimic the folding of $V$ up to this point, but this only happens when either $\alpha=\theta$ or $\beta=\theta$.  Otherwise, if $\beta<\theta$ and $\pi-\alpha<\theta$, or if $\beta>\theta$ and $\pi-\alpha)>\theta$, we will have that $DL(V,\theta)$ mimics the rigid folding of $V$ from the unfolded state up to a point, whereupon the double-line folding angle sum $S$ will return to zero.

\begin{table*}
\caption{Effects of $0<\alpha<\beta<\pi/2$ and $\theta$ on the angle sum $S$ of the major and minor creases in $DL(V,\theta)$.}
\centering
\begin{tabular}{l c c c}
Mode & $S$ spans $[-2\pi,2\pi]$ & $S$ spans $[-M,M]$ for some $M<\pi$ \\
\hline
a-I & $\beta < \theta <\pi-\alpha$ & $\theta <\pi-\alpha, \beta$ or $\theta>\pi-\alpha,\beta$\\
a-II & $\alpha<\theta<\pi-\beta$ & $\theta<\pi-\beta,\alpha$ or $\theta>\pi-\beta, \alpha$\\
b-I & $\alpha<\theta<\beta$ & $\theta<\alpha,\beta$ or $\theta>\alpha,\beta$ \\
b-II & $\pi-\beta<\theta<\pi-\alpha$ & $\theta<\pi-\{\alpha,\beta\}$ or $\theta>\pi-\{\alpha,\beta\}$\\
\hline
\end{tabular}
\label{table1}
\end{table*}

The other folding modes of $DL(V,\theta)$ shown in Figure~\ref{fig:ffd4-summation} follow similar rules; we summarize them in Table~\ref{table1} under the assumption that $0<\alpha<\beta<\pi/2$ and have proven the following:

\begin{theorem}\label{thm:deg4ff-represent}
Among four folding modes of $DL(V,\theta)$ for a flat-foldable, degree-4 vertex $V$,
exactly two modes (a-I and a-II) correspond to one mode (Mode a) of $V$, and the other two modes (b-I and b-II) correspond to the other mode (Mode b) of $V$;
the range of $\theta$ that allows for full range of folding motion and a finite motion for each mode is given by Table~\ref{table1}.
\end{theorem}

The critical value in Table~\ref{table1}, e.g., $\theta = \beta, \pi-\alpha$ for mode a-I, are when one of the fold angles of parallel double lines is kept $0$, so it is a trivial case where the pattern is essentially original $V$ (see Figure~\ref{fig:critical}).
For a practical reason explained in Section~\ref{sec:thick}, we want the ratio between two fold angles of the parallel lines measured by the tangent of half fold angle away from $0:1$ or $1:0$ at critical values of $\theta$.
We call this ratio the \emph{double-line ratio}.
In fact, the double-line ratio of a minor crease can be made completely even, i.e., $1:1$ when $\tan\frac \theta 2$ is the geometric mean of these critical values, e.g.,$\tan \frac\theta 2  = \sqrt{\tan \frac \beta 2 \tan \frac {\pi-\alpha}{2}}$ for mode a-I.
Note that a major crease cannot have a double line ratio of $1:1$, which happens only when $\theta = 0, \pi$.
Table~\ref{table2} shows the range of the double-line ratio that can be obtained by tweaking $\theta$.

\begin{figure}[t]
	\includegraphics[width=0.9\linewidth]{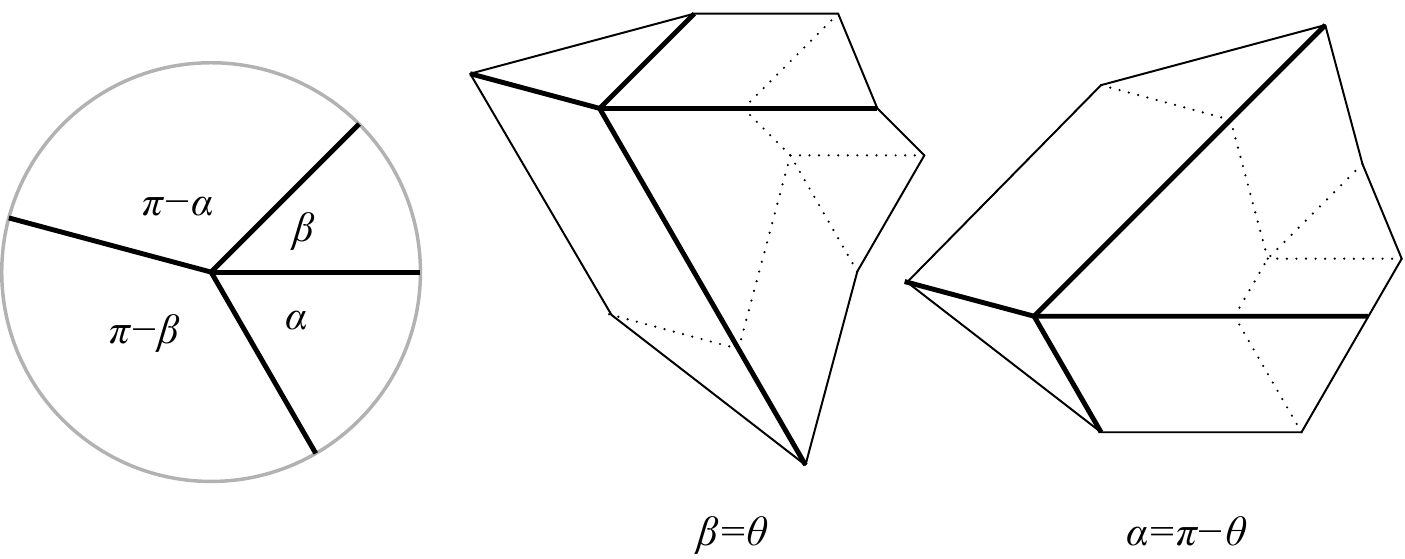}
	\centering
	\caption{Critical case is essentially a single vertex crease pattern.}
	\label{fig:critical}
\end{figure}

\begin{table}
\caption{Range of double-line ratio, written in the order, top:bottom or left:right.}
\centering
\begin{tabular}{l c c }
Mode & Major & Minor\\
\hline
a-I & $1:k$ ($-1<k<1$) & $\neq1:-1$\\
a-II &$k:1$ ($-1<k<1$) & $\neq1:-1$\\
b-I &  $1:k$ ($-1<k<1$) & $\neq1:-1$ \\
b-II & $k:1$ ($-1<k<1$) & $\neq1:-1$\\
\hline
\end{tabular}
\label{table2}
\end{table}

\begin{figure}[t]
	\includegraphics[width=\linewidth]{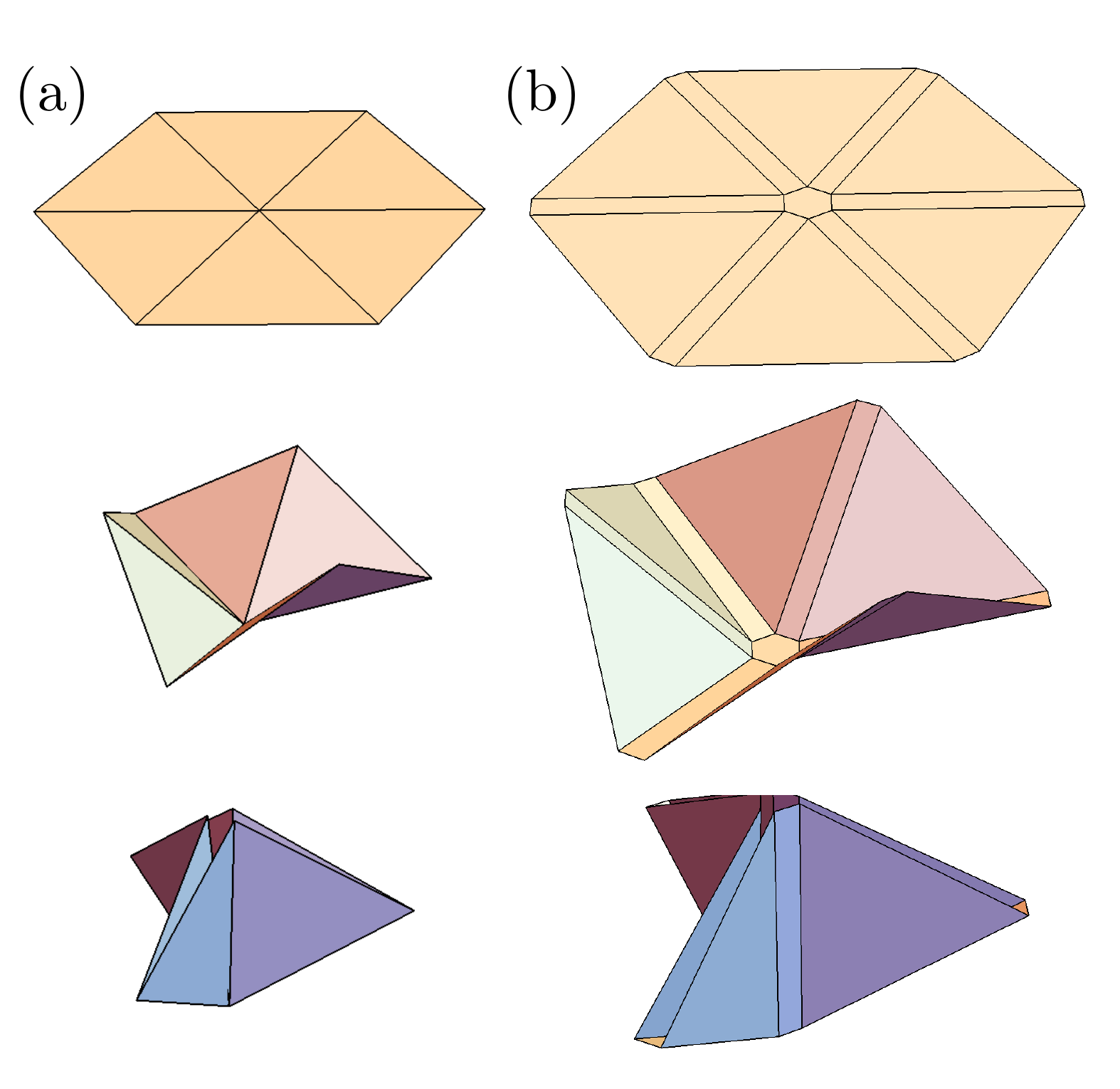}
	\centering
	\caption{(a) A fully-symmetric degree-6 vertex rigidly folding.  (b) The double-line version.}
	\label{fig:deg6}
\end{figure}

\section{Symmetric degree-\lowercase{$2n$} vertex}
\label{sec:2n}

If $V$ is a flat-foldable vertex of degree $2n$ then it is difficult to say anything general about the kinematics of $DL(V,\theta)$, mainly because $V$ will have $2n-3$ DOF, whereas $DL(V,\theta)$ will have one DOF.  Thus the kinematics of $DL(V,\theta)$ will only follow a curve in the configuration space of $V$'s kinematics, and the specifics of this curve will depend on the modes chosen for the vertices of $DL(V,\theta)$.  

Nonetheless, an analysis of the double-line version of such higher-degree vertices can be done on a case-by-case basis.  One case that is especially nice is the \emph{fully-symmetric case}, where all of the sector angles of $V$ are congruent and we assume that $V$ rigidly folds in a symmetric manner:  If $\rho_i$ are the folding angles of the creases of $V$ for $i=0,\ldots, 2n-1$, then we have that $\rho_i=\rho_j$ when the even/odd parity of $i$ and $j$ are the same.

The degree-6 case of $V$ is shown in Figure~\ref{fig:deg6}(a) and the double-line version in (b).  By symmetry, the double-line version will have equal folding angles along each double line, and these pairs alternate mountain-valley.  Let $\rho_1$ denote the valley folding angles and $\rho_2$ the mountain folding angles of the individual double-line creases.  Using the degree-4 kinematics of the Appendix applied to the vertices of $DL(V,90^\circ)$ give us that $-\tan(\rho_2/2) = p(60^\circ,90^\circ)\tan(\rho_1/2)$, and this determines the kinematics shown in Figure~\ref{fig:deg6}(b).

This turns out to prove an interesting fact about the original degree-6 vertex $V$.  Let the valley and mountain folding angles of $V$ be $\rho_a$ and $\rho_b$, respectively.  Since the kinematics of $DL(V,90^\circ)$ must match those of $V$, 
\begin{multline}
\rho_b = 2\rho_2 = -4\arctan\left(p(60^\circ,90^\circ)\tan\frac{\rho_1}{2}\right)\\
\implies \tan\frac{\rho_b}{4} =-p(60^\circ,90^\circ)\tan\frac{\rho_a}{4}
\end{multline}
since $\rho_a = 2\rho_1$.  Therefore the folding angles of a fully-symmetric degree-6 vertex are proportional under reparameterization by tangent of the quarter angle.  This was previously shown in \cite{Tachi-Hull:2016}, but the double-line method gives a different proof.  Furthermore, this same argument works for any fully-symmetric, rigidly-folding degree-$2n$ vertex; the tangent of the quarter of their folding angles will be proportional by $p(180^\circ/n,90^\circ)$.

\begin{figure}[t]
	\includegraphics[width=0.9\linewidth]{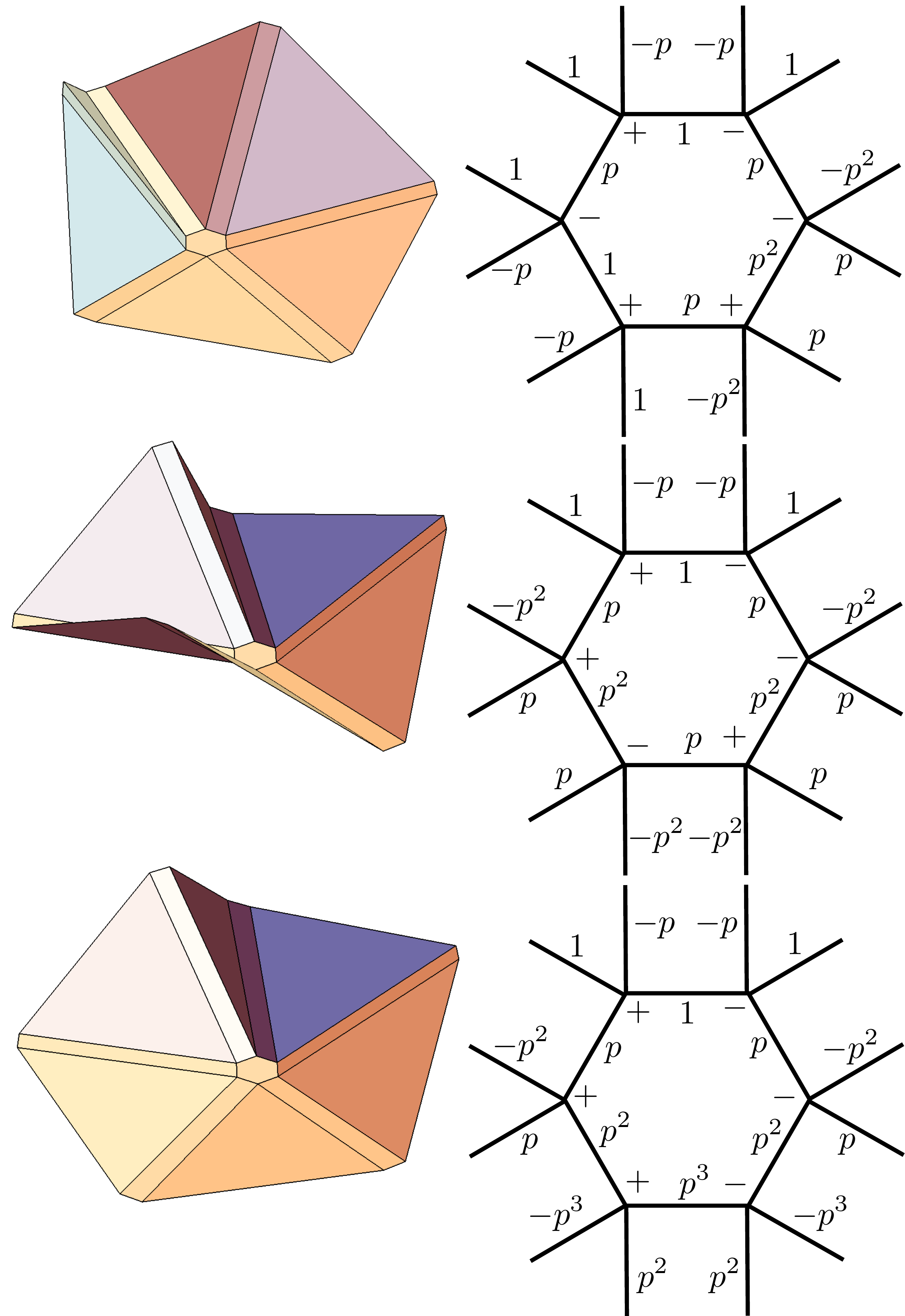}
	\centering
	\caption{The three non-fully-symmetric modes of the all-angles-equal degree-6 double-lined vertex:  Modes $(+ - - + - +)$, $(+ - - + + -)$, and $(- - + + + -)$. The creases are labeled with their folding speeds, where $p=p(60^\circ,90^\circ)$.}
	\label{fig:deg6modes}
\end{figure}

\subsection{Different modes in degree-$6$ case}

The fully-symmetric degree-6 double-line vertex case has folding mode sequence $(+ - + - + -)$ for the vertices of the inner hexagon face.  Other modes are possible, but will break symmetry.  The requirement for the folding modes of $DL(V,90^\circ)$ where $V$ is a degree-6 vertex is that the vertices around the hexagon face have three $+$ modes and three $-$ modes in order for Equation~\eqref{eq:prod} to be satisfied.  Two such sequences of $+$ and $-$ symbols will be equivalent if one is shift of the other, since this is the same as rotating the hexagon.  An easy Burnside's Lemma argument shows that there are only four distinct ways to label the vertices with $+$ and $-$ symbols in this way up to rotational symmetry.  The $(+ - + - + -)$ case is shown in Figure~\ref{fig:deg6}, and the other modes are shown in Figure~\ref{fig:deg6modes}.  Note that since the speed coefficients, which determine the $+$ and $-$ modes of the vertices, are based on a counterclockwise orientation of the hexagon, different mode sequences that are mirror symmetric can produce different modes, as $(+ - - + - +)$ and $(+ - - + + -)$ do in Figure~\ref{fig:deg6modes}.

For a degree-$2n$ vertex $V$ with all sector angles equal to $360^\circ/(2n)$, the number of different modes of $DL(V,90^\circ)$ will be the same as the number of ways to make a $2n$-bead bracelet with $n$ white and $n$ black beads, where flipping the bracelet over is considered different if it changes the order of the colored beads.  This is sequence A003239 in \cite{OEIS} and gives us the following:

\begin{theorem}\label{thm:deg2n-modes}
If $V$ is a symmetric vertex of degree $2n$, then the number of modes of $DL(V,90^\circ)$ is
$$\sum_{d|n} \phi(n/d){2d\choose d}/(2n)$$
where $\phi(x)$ is the Euler phi function, the number of positive integers less than $x$ that are coprime to $x$.
\end{theorem}

Thus a degree-8 symmetric vertex double-lined will have 10 modes, and a degree-12 one will have 26, for example.


\section{Application to thick origami}
\label{sec:thick}

\subsection{thickening}
When realizing deployable structures or robotic systems based on rigid foldable origami, the thickness of the panel is an important issue because a small modification to the position of creases may easily result in loosing rigid foldability.
Several methods for thick panel origami have been proposed~\cite{Hoberman:1988,Tachi:2011,Chen:2015,Ku-Demaine:2016}.
The volume trim method~\cite{Tachi:2011} is a universal basic method applicable to a wide family of rigid foldable origami patterns.
The benefit of the volume trim is that the thick version preserves the original kinematics of ideal zero-thickness origami and that the fabrication is simply realized by sandwiching a zero-thickness core with thick panels.
The drawback is that it often suffers from the trade-off between thickness and the maximum sharpness of fold angles;
to realize a kinematic range defined by a maximum folding angle $\rho_\textrm{max}>\frac\pi 2$, the thickness $t$ of panels is bounded by $ t_\textrm{max} \propto \tan\frac{\pi-\rho_\textrm{max}}{2}$ (Figure~\ref{fig:thick}(2)).
In particular, it is not possible to get a flat folding of $180^\circ$ because that would make the thickness of panels approach to $0$.
This limits the appeal of flat-foldable origami design for compact stowage when real thick panels are used, especially for big structures.

The double-line method can be an effective preconditioning method for volume-trim based thick rigid origami 
because when the parallel foldlines of $DL(V,\theta)$ have the same signs, the double-line method ``chamfers'' theoriginal creases of $V$, where the fold angle of each crease is distributed to the fold angles of the pair of creases (Figure~\ref{fig:thick}(3)).
Therefore the double-line method can split a shaper crease into multiple milder creases that can be realized by reasonably thick panels. 
In particular, we can get a ``flat folded'' state in a macroscopic sense, while using milder fold angles.
Such a use of double lines for each crease for thickening rigid origami is proposed by~\cite{Ku-Demaine:2016} as a crease-offset method, although with the necessity of holes added to each vertex causing a structurally undesirable play.
Our double-line method is essentially a special family of the crease-offset approach that can fill the hole,
thus we may obtain a watertight surface without undesired play and that follows an analytically describable kinematics.

Another benefit of using the double-line method is that we can realize thick panel origami by only placing panes on one side.
This is because the milder fold angle allows for panels to exist on the valley side of the ideal zero-thickness origami.
This is helpful for extending thick rigid origami to layered non-manifold origami because the side without thick panels follows an ideal zero-thickness origami, they can be shared side by side to the other structure.

\begin{figure}[t]
	\includegraphics[width=\linewidth]{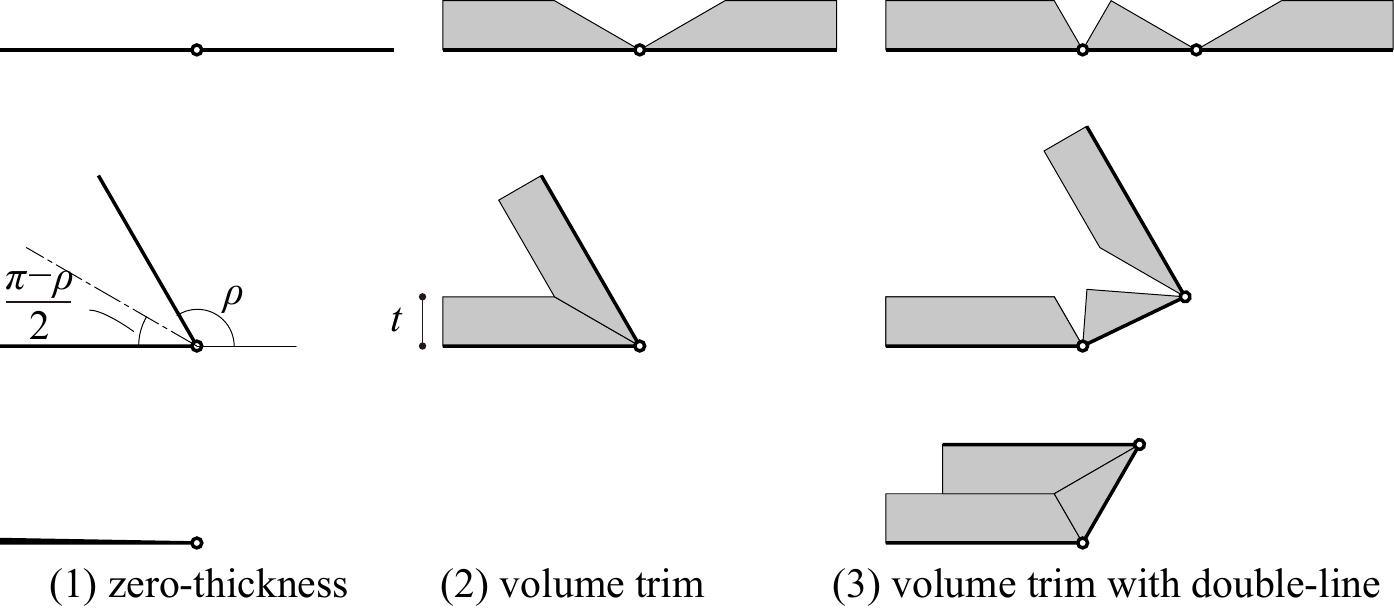}
	\centering
	\caption{Thickening of double-lined origami. Left: ideal zero-thickness origami. Middle: volume trim method applied to the original crease pattern. Right: volume trim method applied to double-lined origami.}
	\label{fig:thick}
\end{figure}

\subsection{Connecting multiple vertices}
Sections~\ref{sec:deg4} and~\ref{sec:2n} show that flat-foldable degree-$4$ vertices and $n$-fold symmetric $2n$-degree vertices can be converted to double-line origami with rigid foldability.
To apply the double-line method for different origami designs, we want to know if this double-line method can apply to origami structures composed of multiple vertices.

Consider connecting $DL(V_1, \theta_1)$ and $DL(V_2, \theta_2)$.
Then we need to consider the compatibility of two fold-lines shared by these vertices.
This can be guaranteed by matching the double-line ratio of shared edges by tweaking parameters $\theta_1$ and $\theta_2$.
More specifically, consider that $\theta_1$ is given, so the double-line ratio of the shared edge is given.
We would like to find $\theta_2$ such that $V_2$ has compatible double-line ratio at the shared edge.
Without the loss of generality, consider $V_2$ folds along mode a, then for either mode a-I or a-II of $DL(V_2,\theta_2)$, there exists $\theta_2$ such that the shared edge has the specified double-line ratio, except when the shared edge is major and the ratio is $1:1$ or when shared edge is minor and the ratio is $1:-1$ (refer to Table~\ref{table2}).
We can avoid such a finite number of bad cases because there is a continuous range of $\theta_1$ that gives a continuous range of the double-line ratio for the shared edge.
In this manner, when we have a serially connected double line vertices $DL(V_1, \theta_1),\dots,DL(V_n, \theta_n)$, another vertex $DL(V_{n+1},\theta_{n+1})$ can be connected to the system when they share only one pair of edges by tweaking $\theta_{n+1}$.
Therefore, there exists rigidly foldable double-line crease pattern of any network of degree-4 flat-foldable vertices that forms a tree (without a cycle, and thus without an interior face).
Furthermore, the double-line pattern can fold along the folding mode of the original crease pattern.
Multi-vertex crease pattern with a cycle of vertices (interior face) in general cannot always give double line rigid origami.
However, some origami tessellations with symmetry, such as Miura-ori and (elongated) Yoshimura-pattern can have double-line version with rigid foldability (Figures~\ref{fig:yoshimura} and~\ref{fig:miura}).

\begin{figure}[t]
	\includegraphics[width=\linewidth]{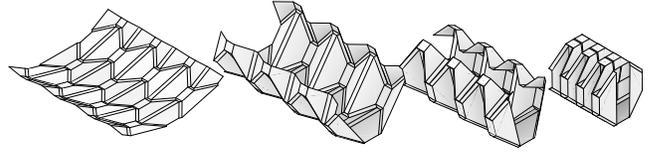}
	\centering
	\caption{Double-line version of elongated Yoshimura pattern.}
	\label{fig:yoshimura}
\end{figure}
\begin{figure}[t]
	\includegraphics[width=\linewidth]{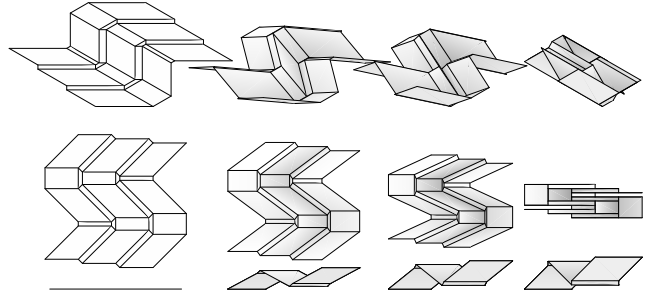}
	\centering
	\caption{Double-line version of Miura-ori. Top: isometric view. Middle and Bottom: top and front views.}
	\label{fig:miura}
\end{figure}

\section{Discussion}
We have shown the concept of double line rigid origami that can be a useful method for interpreting existing rigid origami kinematics, or to obtain a watertight thick rigid origami that can be folded completely flat.
Although we focused the analyses on degree-4 flat foldable vertex and symmetric $2n$ degree vertex in this paper, 
there are more examples of crease patterns that the double-line method can apply; the characterization of such patterns remain unsolved.
Also, we would like to characterize the multi-vertex double-line rigid origami in future.

\section*{Acknowledgments}
T. Tachi is supported by JSPS KAKENHI Grant-in-Aid for Young Scientists (A) 16H06106.

\section*{Appendix}
A flat foldable, degree-$4$ vertex composed of sector angles $\alpha, \beta, \pi-\alpha, \pi-\beta$ in counterclockwise order, rigidly folds along two folding modes keeping the tangents of half fold angles proportional to each other~\cite{huffman:1976,huffman:1978,Tachi-Hull:2016}.
Specifically, let $\rho_i$ denote the fold angle of crease $e_i$, where $i=0,1,2,3$ in the counterclockwise order around the vertex and crease $e_1$ is placed between $\alpha, \beta$.
Then, the folding modes represented by the tangents of half fold angles are
\begin{align}
\left(\begin{matrix} \tan \frac{\rho_0}{2} & \tan \frac{\rho_1}{2} & \tan \frac{\rho_2}{2} &\tan \frac{\rho_3}{2}\end{matrix}\right) =  \nonumber\\
\begin{cases}
\left(\begin{matrix}1 & -p(\alpha,\beta) & 1 & p(\alpha,\beta) \end{matrix}\right) t& \textrm{mode (a)}\\
\left(\begin{matrix}-q(\alpha,\beta)& 1& q(\alpha,\beta)&1 \end{matrix}\right) t&  \textrm{mode (b)}
\end{cases},
\label{eq:folding-modes}
\end{align}
where
\begin{align}
p\left(\alpha,\beta\right)
&={1-\tan{\alpha\over 2}\tan{\beta\over 2}\over 1+\tan{\alpha\over 2}\tan{\beta\over 2}}\\
q\left(\alpha,\beta\right)
&={-\tan{\alpha\over 2} +\tan{\beta\over 2} \over \tan{\alpha\over 2}+\tan{\beta\over 2}}.
\end{align}
\begin{figure}[tbhp]
	\includegraphics[width=\linewidth]{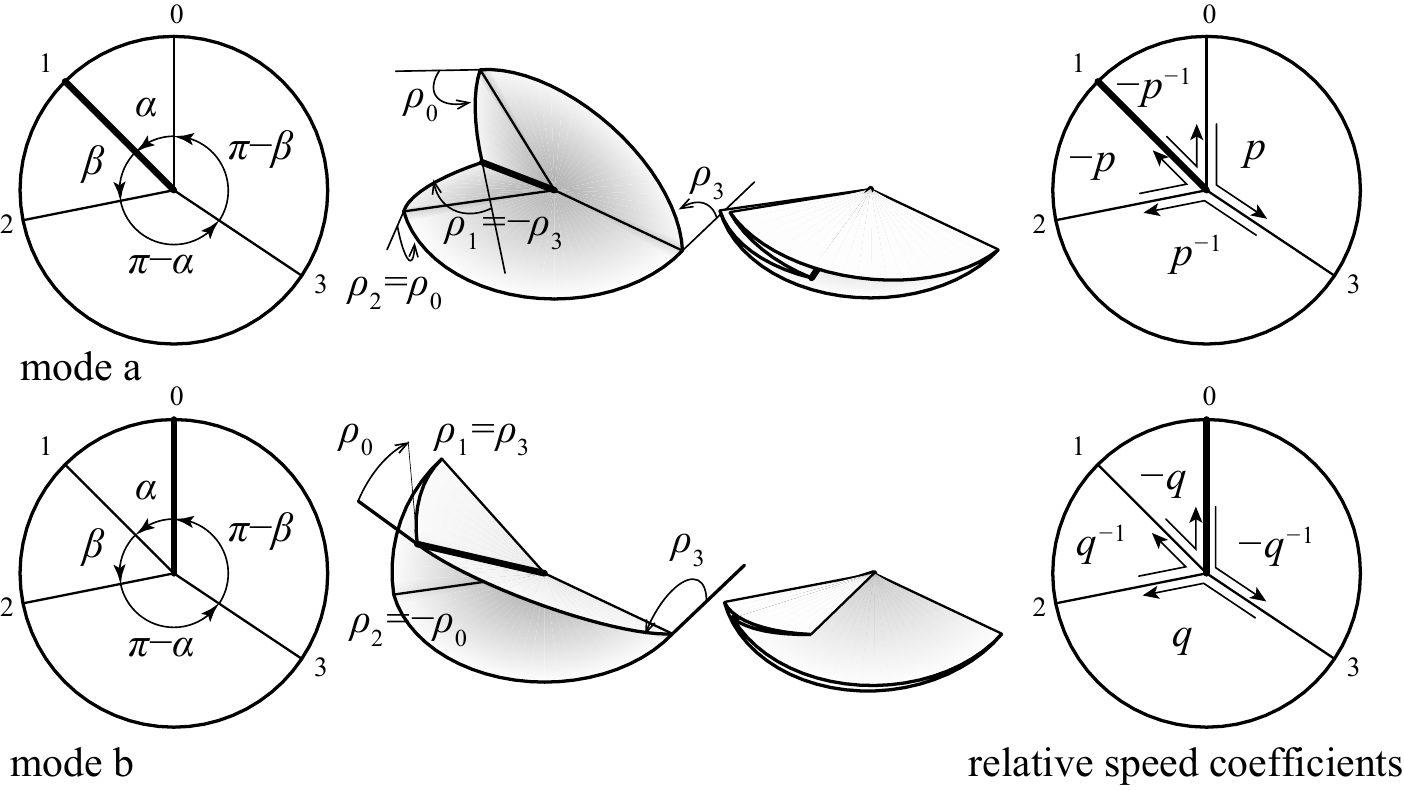}
	\centering
	\caption{A single, flat-foldable degree-4 vertex.}
	\label{fig:ffd4vert}
\end{figure}

Now, we consider connecting these flat-foldable degree-$4$ vertices to compose a multi-vertex origami.
Notice that the ratio between the folding speeds of two edges measured by the tangent of half angle are constant following Equation~\eqref{eq:folding-modes}; we call the constant \emph{speed coefficient}. 
Therefore, the rigid folding motion is given by a consistent folding speed coefficients, and the rigid foldability is the assignment problem of modes at each vertex to switch the coefficients.
Specifically, the following gives the finite rigid foldability of a topologically disk origami composed of flat-foldable degree-$4$ vertices.
For each interior face surrounded by creases $c_i$ ($i = 0,\dots, k-1$) in counterclockwise order,
\begin{align}
\prod_{i=0, \dots, k-1} p_{i+1,i} &= 1,
\label{eq:prod}
\end{align}
where $p_{i+1,i} = \tan \frac{\rho_{i+1}}{2} / \tan \frac{\rho_i}{2}$ are the \emph{speed coefficients} from $c_i$ to $c_{i+1}$, which can be represented by $\pm p(\alpha, \beta)^{\pm 1}$ or $\pm q(\alpha, \beta)^{\pm 1}$ of the vertex incident to $c_i$ and $c_{i+1}$. 
The rightmost column of Figure~\ref{fig:ffd4vert} shows the  speed coefficients between such adjacent creases.

\bibliography{double-line}


\section*{About the authors}
\begin{enumerate}
\item Thomas C. Hull is an associate professor of mathematics at Western New England University who has been studying the mathematics of origami since the 1990s.
\item Tomohiro Tachi is an assistant professor of graphic and computer sciences at the University of Tokyo. ​ ​He has been studying structural morphology and kinematic design through computational origami. 
\end{enumerate}

\end{document}